\documentclass[11pt]{article}
\usepackage{amsfonts,amssymb,amsmath,amsthm}
\usepackage{bbm}
\usepackage{pgf}
\usepackage{xcolor}
\usepackage{hyperref}
\usepackage{enumitem}

\hypersetup{colorlinks=true,citecolor={red}}

\numberwithin{equation}{section}
\def\PP{\mathbb{P}}

\def\11{\mathbbm{1}}

\def\E{\mathbb{E}}
\def\P{\mathbb{P}}
\def\R{\mathbb{R}}

\def\d{\partial}

\def\rM{\rho_{\partial M}}

\newtheorem{thm}{Theorem}[section]
\newtheorem{lem}[thm]{Lemma}

\theoremstyle{remark}
\newtheorem{rem}{Remark}

\begin{document}

\title{Criteria for exponential convergence to quasi-stationary distributions and applications to multi-dimensional diffusions}

\author{Nicolas Champagnat$^{1,2,3}$, Kol\'eh\`e Abdoulaye Coulibaly-Pasquier$^{1,2}$,\\ Denis Villemonais$^{1,2,3}$}

\footnotetext[1]{IECL, Universit\'e de Lorraine, Site de Nancy, B.P. 70239, F-54506 Vandœuvre-lès-Nancy Cedex, France}
\footnotetext[2]{CNRS, IECL, UMR 7502, Vand{\oe}uvre-l\`es-Nancy, F-54506, France}  
\footnotetext[3]{Inria, TOSCA team, Villers-l\`es-Nancy, F-54600, France.\\
  E-mail: Nicolas.Champagnat@inria.fr, Kolehe.Coulibaly@univ-lorraine.fr,\\ Denis.Villemonais@univ-lorraine.fr}

\maketitle

\begin{abstract}
  We consider general Markov processes with absorption and provide criteria ensuring the exponential convergence in total
  variation of the distribution of the process conditioned not to be absorbed. The first one is based
  on two-sided estimates on the transition kernel of the process and the second one on gradient estimates on its semigroup. We apply
  these criteria to multi-dimensional diffusion processes in bounded domains of $\R^d$ or in compact Riemannian manifolds with
  boundary, with absorption at the boundary.
\end{abstract}



\section{Introduction}
\label{sec:intro}

Let $X$ be a Markov process evolving in a measurable state space $E\cup\{\d\}$ absorbed at $\d\notin E$ at time $\tau_\d=\inf\{t\geq
0,\ X_t=\d\}$. We assume that $\P_x(t<\tau_\d)>0$, for all $x\in E$ and all $t\geq 0$, where $\P_x$ is the law of $X$ with initial
position $x$. We consider the problem of existence of a probability measure $\alpha$ on $E$ and of positive constants $B,\gamma>0$
such that, for all initial distribution $\mu$ on $E$,
\begin{align}
  \label{eq:intro-expo-cv}
  \left\|\PP_\mu(X_t\in\cdot\mid t<\tau_\partial)-\alpha(\cdot)\right\|_{TV}\leq B e^{-\gamma t},\quad \forall t\geq 0,
\end{align}
where $\P_\mu$ is the law of $X$ with initial distribution $\mu$ and $\|\cdot\|_{TV}$ is the total variation norm on finite signed
measures. It is well known that~\eqref{eq:intro-expo-cv} entails that $\alpha$ is the unique quasi-stationary distribution for $X$,
that is the unique probability measure satisfying
\begin{align*}
  \alpha(\cdot)=\P_\alpha\left(X_t\in \cdot\mid t<\tau_\d\right),\quad \forall t\geq 0.
\end{align*}

Our goal is to provide sufficient conditions for~\eqref{eq:intro-expo-cv} with applications when $X$ is a diffusion process, absorbed
at the boundary of a domain of $\R^d$ or of a Riemannian manifold. Our first result (Theorem~\ref{thm:two-sided-estimates}) shows
that a two-sided estimate for the transition kernel of a general absorbed Markov process is sufficient to
ensure~\eqref{eq:intro-expo-cv}. This criterion applies in particular to diffusions with smooth coefficients in bounded domains of
$\R^d$ with irregular boundary. Our second result (Theorem~\ref{thm:general-criterion}) concerns Markov processes satisfying gradient
estimates (as in Wang~\cite{wang-04} and Priola and Wang~\cite{Priola2006}), irreducibility conditions and controlled probability of
absorption near the boundary. It applies to diffusions with less regular coefficients in smooth domains of $\R^d$ and to drifted
Brownian motions in compact Riemannian manifolds with $C^2$ boundary.

\bigskip 

Convergence of conditioned diffusion processes have been already obtained for diffusions in domains of $\R^d$, mainly using spectral
theoretic arguments (see for instance \cite{CCLMMS09,Kolb2011,Littin2012,miura-14,HeningKolb,
ChampagnatVillemonais2015c} for $d=1$
and \cite{Cattiaux2008,KnoblochPartzsch2010,DelMoralVillemonais2015} for $d\geq 2$). Among these
references,~\cite{KnoblochPartzsch2010,DelMoralVillemonais2015} give the most general criteria for diffusions in dimension $2$ or
more. Using two-sided estimates and spectral properties of the infinitesimal generator of $X$, Knobloch and
Partzsch~\cite{KnoblochPartzsch2010} proved that~\eqref{eq:intro-expo-cv} holds for a class of diffusion processes evolving in~$\R^d$
($d\geq 3$) with $C^1$ diffusion coefficient, drift in a Kato class and $C^{1,1}$ domain. In~\cite{DelMoralVillemonais2015}, the
authors obtain~\eqref{eq:intro-expo-cv} for diffusions with global Lipschitz coefficients (and additional local regularity near the
boundary) in a domain with $C^2$ boundary. These results can be recovered with our method (see Section~\ref{sec:tse}
and~\ref{sec:particular-case} respectively). When the diffusion is a drifted Brownian motion with drift deriving from a potential,
the authors of~\cite{Cattiaux2008} obtain existence and uniqueness results for the quasi-stationary distribution in cases with
singular drifts and unbounded domains with non-regular boundary that do not enter the settings of this paper.

Usual tools to prove convergence in total variation for processes without absorption involve coupling arguments: for example,
contraction in total variation norm for the non-conditioned semi-group can be obtained using mirror and parallel coupling,
see~\cite{lindvall-rogers-86,wang-04,Priola2006}, or lower bounds on the density of the process that could be obtained for example
using Aronson-type estimates or Malliavin calculus~\cite{aronson,wang,zhang,Nualart2006}. However, on the one hand, lower bounds on transition
densities are not sufficient to control conditional distributions, and on the other hand, the process conditioned not to be killed up
to a given time $t>0$ is a time-inhomogeneous diffusion process with a singular drift for which these methods fail. For instance, a
standard $d$-dimensional Brownian motion $(B_t)_{t\geq 0}$ conditioned not to exit a smooth domain $D\subset\R^d$ up to a time $t>0$
has the law of the solution $(X^{(t)}_s)_{s\in[0,t]}$ to the stochastic differential equation
\begin{align*}
  dX^{(t)}_s=dB_s+\left[\nabla \ln \P_{\cdot}(t-s<\tau_\d)\right](X^{(t)}_s)ds.
\end{align*}
where the drift term is singular near the boundary. Our approach is thus to use the following condition, which is actually equivalent
to the exponential convergence~\eqref{eq:intro-expo-cv} (see~\cite [Theorem~2.1]{champagnat-villemonais-15}).

\paragraph{Condition~(A).} There exist $t_0,c_1,c_2>0$ and a probability measure $\nu$ on $E$ such that 
\begin{itemize}
\item[(A1)]  for all $x\in E$,
  $$
  \PP_x(X_{t_0}\in\cdot\mid t_0<\tau_\partial)\geq c_1\nu(\cdot)
  $$
\item[(A2)]  for all $z\in E$ and all $t\geq 0$,
  $$
  \PP_{\nu}(t<\tau_\partial)\geq c_2\PP_z(t<\tau_\partial).
  $$
\end{itemize}

More precisely, if Condition~(A) is satisfied, then, for all probability measure $\pi$ on $E$,
\begin{align*}
  \left\|\PP_\pi(X_t\in\cdot\mid t<\tau_\partial)-\alpha(\cdot)\right\|_{TV}\leq 2(1-c_1c_2)^{\lfloor t/t_0\rfloor}
\end{align*}
and it implies that, for all probability measures $\pi_1$ and $\pi_2$ on $E$,
\begin{multline}
  \label{eq:intro-contr}
  \left\|\P_{\pi_1}(X_t\in\cdot\mid t<\tau_\d)-\P_{\pi_2}(X_t\in\cdot\mid t<\tau_\d)\right\|_{TV}\\ \leq \frac{(1-c_1c_2)^{\lfloor
      t/t_0\rfloor}}{c(\pi_1)\vee c(\pi_2)}\|\pi_1-\pi_2\|_{TV},
\end{multline}
where $c(\pi_i)=\inf_{t\geq 0}\P_{\pi_i}(t<\tau_\d)/\sup_{z\in E}\P_{z}(t<\tau_\d)$ (see Appendix~\ref{sec:appendix} for a proof of this improvement
of~\cite[Corollary~2.2]{champagnat-villemonais-15}, where the same inequality is obtained with $c(\pi_1)\wedge c(\pi_2)$ instead of
$c(\pi_1)\vee c(\pi_2)$).

Several other properties can also be deduced from Condition~(A). For instance, $e^{\lambda_0 t}\P_x(t<\tau_\d)$
converges when $t\rightarrow+\infty$, uniformly in $x$, to a positive eigenfunction $\eta$ of the infinitesimal
generator of $(X_t,t\geq 0)$ for the eigenvalue $-\lambda_0$ characterized by the relation
$\P_\alpha(t<\tau_\d)=e^{-\lambda_0 t}$, $\forall t\geq 0$~\cite[Proposition~2.3]{champagnat-villemonais-15}. Moreover,
it implies a spectral gap property~\cite[Corollary~2.4]{champagnat-villemonais-15}, the existence and exponential
ergodicity of the so-called $Q$-process, defined as the process $X$ conditioned to never hit the
boundary~\cite[Theorem~3.1]{champagnat-villemonais-15} and a conditional ergodic
property~\cite{ChampagnatVillemonais2016U}. Note that we do not assume that $\P_x(\tau_\d<\infty)=1$, which is only
required in the proofs of~\cite{champagnat-villemonais-15} in order to obtain $\lambda_0>0$. Indeed, the above
inequalities remain true under Condition~(A), even if $\P_x(\tau_\d<+\infty)<1$ for some $x\in E$. The only difference
is that, in this case, $E':=\{x\in E,\ \P_x(\tau_\d<+\infty)=0\}$ is non-empty, $\alpha$ is a classical stationary
distribution such that $\alpha(E')=1$ and $\lambda_0=0$.

The paper is organized as follows. In Section~\ref{sec:tse}, we state and prove a sufficient criterion for~\eqref{eq:intro-expo-cv}
based on two-sided estimates. In Section~\ref{sec:general-result}, we prove~\eqref{eq:intro-expo-cv} for
Markov processes satisfying gradient estimates, irreducibility conditions and controled probability of absorption near the boundary.
In Section~\ref{sec:particular-case}, we apply this result to diffusions in smooth domains of $\R^d$ and to drifted Brownian motions
in compact Riemannian manifolds with smooth boundary. Section~\ref{sec:proof} is devoted to the proof of the criterion of
Section~\ref{sec:general-result}. Finally, Appendix~\ref{sec:appendix} gives the proof of~\eqref{eq:intro-contr}.

\section{Quasi-stationary behavior under two-sided estimates}
\label{sec:tse}

In this section, we consider as in the introduction a general absorbed Markov process $X$ in $E\cup\{\d\}$ satisfying two-sided
estimates: there exist a time $t_0>0$, a constant $c>0$, a positive measure $\mu$ on $E$ and a measurable function
$f:E\rightarrow(0,+\infty)$ such that
\begin{align}
  \label{eq:tse}
  c^{-1}f(x)\mu(\cdot)\leq \P_x(X_{t_0}\in\cdot)\leq cf(x)\mu(\cdot),\ \forall x\in E.
\end{align}
Note that this implies that $f(x)\mu(E)\leq c$ for all $x\in E$, hence $\mu$ is finite and $f$ is bounded. As a consequence, one can
assume without loss of generality that $\mu$ is a probability measure and then $\|f\|_\infty\leq c$. Note also that $f(x)>0$ for all
$x\in E$ entails that $\P_x(t_0<\tau_\d)>0$ for all $x\in E$ and hence, by Markov property, that $\P_x(t<\tau_\d)>0$ for all $x\in E$
and all $t>0$, as needed to deduce~\eqref{eq:intro-expo-cv} from Condition~(A) (see~\cite{champagnat-villemonais-15}).

Estimates of the form~\eqref{eq:tse} are well known for diffusion processes in a bounded domain of $\R^d$ since the seminal paper of
Davies and Simon~\cite{davies-simon-84}. The case of standard Brownian motion in a bounded $C^{1,1}$ domain of $\R^d$, $d\geq 3$ was
studied in~\cite{zhang-02}. This result has then been extended in~\cite{kim-song-07} to diffusions in a bounded $C^{1,1}$ domain in
$\R^d$, $d\geq 3$, with infinitesimal generator
$$
L=\frac{1}{2}\sum_{i,j=1}^da_{ij}\partial_{i}\partial_j+\sum_{i=1}^d b_i\partial_i,
$$
with symmetric, uniformly elliptic and $C^1$ diffusion matrix $(a_{ij})_{1\leq i,j\leq d}$, and with drift $(b_i)_{1\leq i\leq d}$ in
the Kato class $K_{d,1}$, which contains $L^p(dx)$ functions for $p>d$. Diffusions on bounded, closed Riemannian manifolds with
irregular boundary and with generator
$$
L=\Delta+X,
$$
where $\Delta$ is the Laplace-Beltrami operator and $X$ is a smooth vector field, were also studied in~\cite{lierl-saloffcoste-14}.
Two-sided estimates are also known for processes with jumps~\cite{chen-al-12,bogdan-al-10,chen-al-11,kim-kim-14,chen-al-15}.

\begin{thm}
  \label{thm:two-sided-estimates}
  Assume that there exist a time $t_0>0$, a constant $c>0$, a probability measure $\mu$ on $E$ and a measurable function
  $f:E\rightarrow(0,+\infty)$ such that~\eqref{eq:tse} holds. Then Condition~(A) is satisfied with $\nu=\mu$, $c_1=c^{-2}$ and
  $c_2=c^{-3}\mu(f)$. In addition, for all probability measures $\pi_1$ and $\pi_2$ on $E$, we have 
  \begin{multline}
    \label{eq:contraction}
    \left\|\P_{\pi_1}(X_t\in\cdot\mid t<\tau_\d)-\P_{\pi_2}(X_t\in\cdot\mid t<\tau_\d)\right\|_{TV}\\ \leq c^3
    \frac{(1-c^{-5}\mu(f))^{\lfloor t/t_0\rfloor}}{\pi_1(f)\vee \pi_2(f)}\|\pi_1-\pi_2\|_{TV},
  \end{multline}
  Moreover, the unique quasi-stationary distribution $\alpha$ for $X$ satisfies
  \begin{align}
    \label{eq:borne-QSD}
    c^{-2}\mu\leq\alpha\leq c^{2}\mu.
  \end{align}
\end{thm}

\begin{rem}
  \label{rem:spectral-gap}
  Recall that to any quasi-stationary distribution $\alpha$ is associated an eigenvalue $-\lambda_0\leq 0$. We deduce from the
  two-sided estimate~\eqref{eq:tse} and~\cite[Corollary~2.4]{champagnat-villemonais-15} an explicit estimate on the second spectral
  gap of the infinitesimal generator $L$ of $X$ (defined as acting on bounded measurable functions on $E\cup\{\d\}$): for all
  $\lambda$ in the spectrum of $L$ such that $\lambda\notin\{0,\lambda_0\}$, the real part of $\lambda$ is smaller than
  $-\lambda_0+t_0^{-1}\log(1-c^{-5}\mu(f))$.
\end{rem}

\begin{rem}
  \label{rem:KP}
  In particular, we recover the results of Knobloch and Partzsch~\cite{KnoblochPartzsch2010}. They proved
  that~\eqref{eq:intro-expo-cv} holds for a class of diffusion processes evolving in~$\R^d$ ($d\geq 3$), assuming continuity of the
  transition density, existence of ground states and the existence of a two-sided estimate involving the ground states of the
  generator. Similar results were obtained in the one-dimensional case in~\cite{miura-14}.
\end{rem}

\begin{proof}[Proof of Theorem~\ref{thm:two-sided-estimates}]
We deduce from~\eqref{eq:tse} that, for all $x\in E$,
\begin{align}
  \label{eq:bound-QSD}
  c^{-2}\,\mu(\cdot)\leq \P_x(X_{t_0}\in\cdot\mid t_0<\tau_\d)=\frac{\P_x(X_{t_0}\in\cdot)}{\P_x(X_{t_0}\in E)} \leq c^2\mu(\cdot).
\end{align}
We thus obtain (A1) with $c_1=c^{-2}$ and $\nu=\mu$.

Moreover, for any probability measure $\pi$ on $E$ and any $z\in E$,
\begin{align*}
  \P_\pi(X_{t_0}\in\cdot)&\geq c^{-1}\pi(f)\mu(\cdot)\\
  &\geq \frac{f(z)}{\|f\|_\infty}c^{-1}\pi(f)\mu(\cdot)\\
  &\geq c^{-3}\pi(f)\P_z(X_{t_0}\in\cdot).
\end{align*}
Hence, for all $t\geq t_0$, we have by Markov's property
\begin{align*}
  \P_\pi(t<\tau_\d)&=\E_{\pi}\left(\P_{X_{t_0}}(t-t_0<\tau_\d)\right)\\
  &\geq c^{-3}\pi(f)\, \E_z\left(\P_{X_{t_0}}(t-t_0<\tau_\d)\right)\\
  &= c^{-3}\pi(f)\, \P_z(t<\tau_\d).
\end{align*}
When $t\leq t_0$, we have $\P_\pi(t<\tau_\d)\geq \P_\pi(t_0<\tau_\d) \geq c^{-1}\pi(f)\geq c^{-3}\pi(f)$ and hence
$\P_\pi(t<\tau_\d)\geq c^{-3}\pi(f)\P_z(t<\tau_\d)$, so that
$$
c(\pi):=\inf_{t\geq 0}\frac{\P_{\pi}(t<\tau_\d)}{\sup_{z\in E}\P_{z}(t<\tau_\d)}\geq c^{-3}\pi(f).
$$
Taking $\pi=\nu=\mu$, this entails (A2) for $c_2=c^{-3}\mu(f)$ and~\eqref{eq:intro-contr} implies~\eqref{eq:contraction}. The
inequality~\eqref{eq:borne-QSD} then follows from~\eqref{eq:bound-QSD}.
\end{proof}

\section{Quasi-stationary behavior under gradient estimates}
\label{sec:main}

In this section, we explain how gradient estimates on the semi-group of the Markov process $(X_t,t\geq 0)$ imply the exponential
convergence~\eqref{eq:intro-expo-cv}.

\subsection{A general result}
\label{sec:general-result}

We assume that the process $X$ is a strong Markov, continuous\footnote{The assumption of continuity is only used to ensure that the
  entrance times in compact sets are stopping times for the natural filtration (cf.\ e.g.~\cite[p.\,48]{le-gall}), and hence that the
  strong Markov property applies at this time. Our result would also hold true for c\`adl\`ag (weak) Markov processes provided that
  the strong Markov property applies at the hitting times of compact sets.} process and we assume that its state
space $E\cup\{\d\}$ is a compact metric space with metric $\rho$ equipped with its Borel $\sigma$-field. Recall that $\d$ is absorbing
and that we assume that $\P_x(t<\tau_\d)>0$ for all $x\in E$ and $t\geq 0$. Our result holds true under three conditions: first, we
assume that there exists $t_1>0$ such that the process satisfies a gradient estimate of the form: for all bounded measurable function
$f:E\cup\{\d\}\rightarrow\R$
\begin{equation}
  \label{eq:wang-g}
  \|\nabla P_{t_1}f\|_\infty\leq C\|f\|_\infty,
\end{equation}
where $P_tf(x)=\E_x(f(X_t)\11_{t<\tau_\d})$ denotes the Dirichlet semi-group of $X$ and the (a bit informal in such a general
setting) notation $\|\nabla P_{t_1}f\|_\infty$ has to be understood as
$$
\|\nabla P_{t_1}f\|_\infty:=\sup_{x,y\in E\cup\{\d\}}\frac{|P_{t_1}f(x)-P_{t_1}f(y)|}{\rho(x,y)}.
$$
Second, we assume that there exist a compact subset $K$ of $E$ and a constant $C'>0$ such that, for all $x\in E$,
\begin{align}
\label{eq:majoration-g}
\P_x(T_{K} \leq t_1 <\tau_\d)\geq C'\,\rho_{\d}(x),
\end{align}
where $\rho_{\d}(x):=\rho(x,\d)$ and $T_{K}=\inf\{t\geq 0,\ X_t\in K\}$. Finally, we need the following irreducibility condition:
for all $x,y\in E$ and all $r>0$,
\begin{equation}
  \label{eq:irreducibility-g}
  \P_x(X_s\in B(y,r),\ \forall s\in[t_1,2t_1])>0,
\end{equation}
where $B(y,r)$ denotes the ball of radius $r$ centered at $y$.

\begin{thm}
  \label{thm:general-criterion}
  Assume that the process $(X_t,t\geq 0)$ satisfies~\eqref{eq:wang-g},~\eqref{eq:majoration-g} and~\eqref{eq:irreducibility-g} for
  some constant $t_1>0$. Then Condition~(A) and hence~\eqref{eq:intro-expo-cv} are satisfied. Moreover, there exist two constants
  $B,\gamma>0$ such that, for any initial distributions $\mu_1$ and $\mu_2$ on $E$,
  \begin{multline}
    \label{eq:thm:general-criterion}
    \left\|\P_{\mu_1}(X_t\in\cdot\mid t<\tau_\d)-\P_{\mu_2}(X_t\in\cdot\mid t<\tau_\d)\right\|_{TV}\\
    \leq \frac{Be^{-\gamma t}}{\mu_1(\rho_{\d})\vee \mu_2(\rho_{\d})}\|\mu_1-\mu_2\|_{TV}.
  \end{multline}
\end{thm}

The proof of this result is given in Section~\ref{sec:proof}.

\subsection{The case of diffusions in compact Riemannian manifolds}
\label{sec:particular-case}

In this section, we provide two sets of assumptions for diffusions in compact manifolds with boundary $M$ absorbed at the boundary
$\d M$ (i.e.\ $E=M\setminus \d M$ and $\d=\{\d M\}$) to which the last theorem applies:
\begin{enumerate}
\item[S1.] $M$ is a bounded, connected and closed $C^2$ Riemannian manifold with $C^2$ boundary $\d M$ and the infinitesimal
  generator of the diffusion process $X$ is given by $L=\frac{1}{2}\Delta+Z$, where $\Delta$ is the Laplace-Beltrami operator and $Z$
  is a $C^1$ vector field.
\item[S2.] $M$ is a compact subset of $\R^d$ with non-empty, connected interior and $C^2$ boundary $\d M$ and $X$ is solution to the
  SDE $dX_t=s(X_t)dB_t+b(X_t)dt$, where $(B_t,t\geq 0)$ is a $r$-dimensional Brownian motion, $b:M\rightarrow\R^d$ is bounded and
  continuous and $s:M\rightarrow \R^{d\times r}$ is continuous, $ss^*$ is uniformly elliptic and for all $r>0$,
  \begin{align}
    \sup_{x,y\in M,\ |x-y|=r}\frac{|s(x)-s(y)|^2}{r}\leq g(r)
    \label{eq:hyp-sigma-priola-wang} 
  \end{align}
  for some function $g$ such that $\int_0^1 g(r)dr<\infty$.
\end{enumerate}

Note that~\eqref{eq:hyp-sigma-priola-wang} is satisfied as soon as $s$ is uniformly $\alpha$-H\"older on $M$ for some $\alpha>0$. 

Let us now check that Theorem~\ref{thm:general-criterion} applies in both situations.

First, the gradient estimate~\eqref{eq:wang-g} is satisfied (see Wang in~\cite{wang-04} and Priola and Wang in~\cite{Priola2006},
respectively). These two references actually give a stronger version of~\eqref{eq:wang-g}:
\begin{equation}
\label{eq:gradEst}
  \|\nabla P_{t}f\|_\infty\leq \frac{c}{1\wedge\sqrt{t}}\,\|f\|_\infty,\quad\forall t>0.
\end{equation}
The set of assumptions S2 is not exactly the same as in~\cite{Priola2006}, but they clearly imply (i), (ii), (iv) of~\cite[Hyp.
4.1]{Priola2006} (see~\cite[Lemma 3.3]{Priola2006} for the assumption on $s$) and, since we assume that $M$ is bounded and $C^2$,
assumptions (iii') and (v) are also satisfied (see~\cite[Rk.\ 4.2]{Priola2006}). Moreover, the gradient estimate of~\cite{Priola2006} is stated for $x\in M\setminus \d M\mapsto P_t f(x)$, but can be easily extended to $x\in M$ since $P_t f(x)\rightarrow 0$ when $x\rightarrow \d M$. Note also that in both references, the gradient
estimates are obtained for not necessarily compact manifolds.

The irreducibility assumption~\eqref{eq:irreducibility-g} is an immediate consequence of classical support theorems~\cite[Exercise
6.7.5]{stroock-varadhan-79} for any value of $t_1>0$.

It only remains to prove the next lemma.

\begin{lem}
  \label{lem:majoration}
  There exist $t_1,\varepsilon,C'>0$ such that, for all $x\in M$,
  \begin{align}
    \label{eq:majoration}
    \P_x(T_{\varepsilon} \leq t_1 <\tau_\d)\geq C'\,\rho_{\d M}(x),
  \end{align}
  where $\rho_{\d M}(x)$ is the distance between $x$ and $\d M$, $T_{\varepsilon}=\inf\{t\geq 0,\ X_t\in M_{\varepsilon}\}$ and the
  compact set $M_\varepsilon$ is defined as $\{x\in M:\rho_{\d M}(x)\geq\varepsilon\}$.
\end{lem}

\begin{proof}[Proof of Lemma~\ref{lem:majoration}]
Let $\varepsilon_0>0$ be small enough for $\rho_{\d M}$ to be $C^2$ on $M\setminus M_{\varepsilon_0}$. For all $t<
T_{\varepsilon_0}$, we define $ Y_t=\rM(X_t). $ In both situations S1 and S2, we have
\begin{align*}
dY_t= \sigma_t dW_t+b_t dt,
\end{align*}
where $W$ is a standard Brownian motion, where $\sigma_t\in[\underline{\sigma},\bar{\sigma}]$ and $|b_t|\leq\bar{b}$ are adapted
continuous processes, with $0<\underline{\sigma},\bar{\sigma},\bar{b}<\infty$. There exists a differentiable time-change $\tau(s)$ 
such that $\tau(0)=0$ and 
\begin{align*}
\widetilde{W}_s:=\int_0^{\tau(s)} \sigma_t dW_t
\end{align*}
is a Brownian motion and $\tau'(s)\in [\bar{\sigma}^{-2},\underline{\sigma}^{-2}]$. In addition, 
\begin{align*}
\int_0^{\tau(s)} b_t\,dt \geq -\bar{b} \tau(s)\geq -\bar{b} \underline{\sigma}^{-2} s.
\end{align*}
As a consequence, setting $Z_s=Y_{0}+\widetilde{W}_s-\bar{b} \underline{\sigma}^{-2} s$, we have almost surely $Z_s\leq Y_{\tau(s)}$ for all $s$
such that $\tau(s)\leq T_{\varepsilon_0}$. 

Setting $a=\bar{b} \underline{\sigma}^{-2}$, the function
$$
f(x)= \frac{e^{2ax}-1}{2a}
$$
is a scale function for the drifted Brownian motion $Z$. The diffusion process defined by $N_t=f(Z_{t})$ is a martingale and its
speed measure is given by $s(dv)=\frac{dv}{(1+2av)^2}$. The Green formula for one-dimensional diffusion processes~\cite[Lemma
23.10]{kallenberg-02} entails, for $\varepsilon_1=f(\varepsilon_0)$ and all $u\in(0,\varepsilon_1/2)$ (in the following lines,
$\P_u^N$ denotes the probability with respect to $N$ with initial position $N_0=u$),
\begin{align}
  \P_u^N(t\leq T^N_0\wedge T^N_{\varepsilon_1/2})&\leq \frac{\E^N_u(T^N_0\wedge T^N_{\varepsilon_1/2})}{t}=
  \frac{2}{t}\,\int_0^{\varepsilon_1/2} \left(1-\frac{u\vee v}{\varepsilon_1/2}\right)(u\wedge v)s(dv)\notag\\
  &\leq u\,\frac{C_{\varepsilon_1}}{t},\ \text{ where $C_{\varepsilon_1}=2\,\int_0^{\varepsilon_1/2}\frac{dv}{(1+2av)^2}$,}\label{eq:previneq}
\end{align}
where we set $T^N_\varepsilon=\inf\{t\geq 0,\ N_t=\varepsilon\}$. Let us fix $s_1=\varepsilon_1 C_{\varepsilon_1}$. Since $N$ is a
martingale, we have, for all $u\in(0,\varepsilon_1/2)$,
\begin{align*}
  u=\E^N_u(N_{s_1\wedge T^N_{\varepsilon_1/2}\wedge T^N_{0}})&\leq \frac{\varepsilon_1}{2} \P_u^N(T^N_{\varepsilon_1/2}\leq s_1\wedge
  T^N_0)+\frac{\varepsilon_1}{2} \P^N_u(s_1<T^N_{\varepsilon_1/2}\wedge T^N_0)\\
  &\leq \frac{\varepsilon_1}{2} \P^N_u(T^N_{\varepsilon_1/2}\leq s_1\wedge T^N_0)+\frac{u}{2}.
\end{align*}
Hence there exists a constant $A>0$ such that $\P^N_u(T^N_{\varepsilon_1/2}\leq s_1\wedge T^N_0) \geq A\,u$, or, in other words,
\begin{align*}
  \P_x(T^Z_{\varepsilon}\leq\underline{\sigma}^2 t_1\wedge T^Z_0) \geq A\,f(\rho_{\d M}(x))\geq A\, \rho_{\d M}(x)
\end{align*}
for all $x\in M\setminus M_{\varepsilon}$, where $t_1=s_1 \underline{\sigma}^{-2}$ and $\varepsilon=f^{-1}(\varepsilon_1/2)$.

Now, using the fact that the derivative of the time change $\tau(s)$ belongs to $[\bar{\sigma}^{-2},\underline{\sigma}^{-2}]$ and
that $Z_s\leq Y_{\tau(s)}$, it follows that for all $x\in M\setminus M_{\varepsilon}$,
\begin{align*}
  \P_x(T_{\varepsilon}^Y \leq t_1\wedge T_0^Y) &\geq\P_x (T^Z_{\varepsilon}\leq\underline{\sigma}^2 t_1\wedge T^Z_0) \geq A {\rho_{\d
      M}(x)}.
\end{align*}
Therefore,
\begin{align*}
  \P_x(T_{\varepsilon}^Y\leq t_1< T_0^Y)&\geq \E_x\left[\11_{T_{\varepsilon}^Y\leq t_1\wedge T_0^Y}\,\P_{X_{T_{\varepsilon}^Y}}(t_1<\tau_\d)\right]\\
  &\geq \P_x (T_{\varepsilon}^Y\leq t_1\wedge T_0^Y)\,\inf_{y\in M_{\varepsilon}}\P_{y}(t_1<\tau_\d) 
  \geq C' {\rho_{\d M}(x)},
\end{align*}
where we used that $\inf_{y\in M_{\varepsilon}}\P_{y}(t_1<\tau_\d)>0$. This last fact follows from the inequality
$\P_{y}(t_1<\tau_\d)>0$ for all $y\in M\setminus\d M$, consequence of~\eqref{eq:irreducibility-g} and from the Lipschitz-continuity
of $y\mapsto \P_{y}(t_1<\tau_\d)=P_{t_1}\11_E(y)$, consequence of~\eqref{eq:gradEst}.

Finally, since $T_{\varepsilon}=0$ under $\P_x$ for all $x\in M_{\varepsilon}$, replacing $C'$ by $C'\wedge[\inf_{y\in M_{\varepsilon}}\P_{y}(t_1<\tau_\d)/\text{diam}(M)]$
entails~\eqref{eq:majoration} for all $x\in M$.
\end{proof}

\begin{rem}
  \label{rem:with-killing}
  The gradient estimates of~\cite{Priola2006} are proved for diffusion processes with space-dependent killing
  rate $V:M\rightarrow[0,\infty)$. More precisely, they consider infinitesimal generators of the form
  \begin{align*}
    L=\frac{1}{2}\sum_{i,j=1}^d[ss^*]_{ij}\partial_{i}\partial_j+\sum_{i=1}^d b_i\partial_i-V
  \end{align*}
  with $V$ bounded measurable. Our proof also applies to this setting.
\end{rem}

\begin{rem}
  \label{rem:aronson}
  We have proved in particular that Condition~(A1) is satisfied in situations S1 and S2. This is a minoration of conditional
  distributions of the diffusion. For initial positions in compact subsets of $M\setminus\d M$, this reduces to a lower bound for the
  (unconditioned) distribution of the process. Such a result could be obtained from density lower bounds using number of techniques,
  for example Aronson-type estimates~\cite{aronson,wang,zhang} or continuity properties~\cite{Dynkin2002}. Note that our result does
  not rely on such techniques, since it will appear in the proof that Conditions~\eqref{eq:wang-g} and~\eqref{eq:irreducibility-g}
  are sufficient to obtain $\P_x(X_{t_0}\in\cdot)\geq \widetilde{\nu}$ for all $x\in M_\varepsilon$ for some positive measure
  $\widetilde{\nu}$.
\end{rem}

\subsection{Proof of Theorem~\ref{thm:general-criterion}}
\label{sec:proof}

The proof is based on the following equivalent form of Condition~(A) (see~\cite[Thm.\,2.1]{champagnat-villemonais-15})

\paragraph{Condition~(A').} There exist $t_0,c_1,c_2>0$ such that 
\begin{itemize}
\item[(A1')]  for all $x,y\in E$, there exists a probability measure $\nu_{x,y}$ on $E$ such that
  $$
  \PP_x(X_{t_0}\in\cdot\mid t_0<\tau_\partial)\geq c_1\nu_{x,y}(\cdot)\quad\text{and}\quad\PP_y(X_{t_0}\in\cdot\mid t_0<\tau_\partial)\geq
  c_1\nu_{x,y}(\cdot)
  $$
\item[(A2')]  for all $x,y,z\in E$ and all $t\geq 0$,
  $$
  \PP_{\nu_{x,y}}(t<\tau_\partial)\geq c_2\PP_z(t<\tau_\partial).
  $$
\end{itemize}

Note that (A1') is a kind of coupling for conditional laws of the Markov process starting from different initial conditions. It is
thus natural to use gradient estimates to prove such conditions since they are usually obtained by coupling of the paths of the
process (see~\cite{wang-04,Priola2006}).

We divide the proof into four steps. In the first one, we obtain a lower bound for $\P_x(X_{2t_1}\in K\mid 2t_1<\tau_\d)$. The
second and third ones are devoted to the proof of (A1') and (A2'), respectively. The last one gives the proof
of~\eqref{eq:thm:general-criterion}.

\subsubsection{Return to a compact conditionally on non-absorption}
\label{sec:step-1}

The gradient estimate~\eqref{eq:wang-g} applied to $f=\11_E$ implies that $P_{t_1}\11_E$ is Lipschitz. Since
$\P_\d(t_1<\tau_\d)=0$, we obtain, for all $x\in E$,
\begin{align}
  \label{eq:minoration}
  \P_x(t_1<\tau_\d)\leq C\,\rho_{\d}(x).
\end{align}
Combining this with Assumption~\eqref{eq:majoration-g}, we deduce that, for all $x\in E$,
\begin{align*}
  \PP_x(T_{K}\leq t_1\mid t_1<\tau_\d) & =\frac{\PP_x(T_{K}\leq t_1<\tau_\d)}{\PP_x(t_1<\tau_\d)}\geq \frac{C'}{C}.
\end{align*}
Fix $x_0\in K$ and let $r_0=d(x_0,\d)/2$.  We can assume without loss of
generality that $B(x_0,r_0)\subset K$, since
Assumption~\eqref{eq:majoration-g} remains true if one replaces $K$ by
the set $K\cup \overline{B(x_0,r_0)}\subset E$ (which is also a compact set, as
a closed subset of the compact set $E\cup\{\d\}$). Then, it follows from~\eqref{eq:irreducibility-g} that, for all $x\in E$,
$$
\E_x\left[\P_{X_{t_1}}(X_s\in B(x_0,r_0),\ \forall s\in [0,t_1])\right]=\P_x(X_s\in B(x_0,r_0),\ \forall s\in[t_1,2t_1])>0.
$$
Because of~\eqref{eq:wang-g}, the left-hand side is continuous w.r.t.\ $x\in M$, and hence
$$
\inf_{x\in K}\P_x(X_s\in K,\ \forall s\in[t_1,2t_1])\geq \inf_{x\in K}\P_x(X_s\in B(x_0,r_0),\ \forall s\in[t_1,2t_1])>0.
$$
Therefore, it follows from the strong Markov property at time $T_K$ that
\begin{align*}
  \P_x(X_{2t_1}\in K\mid 2t_1<\tau_\d) & \geq\frac{\P_x(X_{2t_1}\in K)}{\P_x(t_1<\tau_\d)} \\
  & \geq\frac{\P_x(T_K\leq t_1\text{ and }X_{T_K+s}\in K,\ \forall s\in[t_1,2t_1])}{\P_x(t_1<\tau_\d)} \\
  & \geq \inf_{x\in K}\P_x(X_s\in K,\ \forall s\in[t_1,2t_1])\frac{\P_x(T_K\leq t_1)}{\P_x(t_1<\tau_\d)}.
\end{align*}
Therefore, we have proved that, for all $x\in E$,
\begin{equation}
  \label{eq:step1}
  \P_x(X_{2t_1}\in K \mid 2t_1<\tau_\d)\geq A,
\end{equation}
for the positive constant $A:=\inf_{x\in K}\P_x(X_s\in K,\ \forall s\in[t_1,2t_1]) C'/C$.

\subsubsection{Proof of (A1')}

For all $x,y\in E$, let $\mu_{x,y}$ be the infimum measure of $\delta_x P_{2t_1}$ and $\delta_y P_{2t_1}$, i.e.\ for all measurable
$A\subset E$,
$$
\mu_{x,y}(A):=\inf_{A_1\cup A_2=A,\ A_1,A_2\text{ measurable}}(\delta_x P_{2t_1}\mathbbm{1}_{A_1}+\delta_y P_{2t_1}\mathbbm{1}_{A_2}).
$$
The proof of (A1') is based on the following lemma.

\begin{lem}
  \label{lem:positivity}
  For all bounded continuous function $f:E\rightarrow\R_+$ not identically 0, the function $(x,y)\in E^2\mapsto \mu_{x,y}(f)$ is
  Lipschitz and positive.
\end{lem}

\begin{proof}
By~\eqref{eq:wang-g}, for all bounded measurable $g:E\rightarrow\R$,
\begin{equation}
  \label{eq:wang2}
  \|\nabla P_{2t_1}g\|_\infty=\|\nabla P_{t_1}(P_{t_1}g)\|_\infty\leq C\|P_{t_1}g\|_\infty\leq C\|g\|_{\infty}.
\end{equation}
Hence, for all $x,y\in E$,
\begin{equation}
  |P_{2t_1} g(x)-P_{2t_1}g(y)|\leq C\|g\|_\infty\rho(x,y).  
\end{equation}
This implies the uniform Lipschitz-continuity of $P_{2t_1}g$. In particular, we deduce that
$$
\mu_{x,y}(f)=\inf_{A_1\cup A_2=E}\left\{P_{2t_1}(f\mathbbm{1}_{A_1})(x)+P_{2t_1}(f\mathbbm{1}_{A_2})(y)\right\}
$$
is continuous w.r.t.\ $(x,y)\in E^2$ (and even Lipschitz).

Let us now prove that $\mu_{x,y}(f)>0$. Let us define $\bar{\mu}_{x,y}$ as the infimum measure of $\delta_x P_{t_1}$ and $\delta_y
P_{t_1}$: for all measurable $A\subset E$,
$$
\bar{\mu}_{x,y}(A):=\inf_{A_1\cup A_2=A}(\delta_x P_{t_1}\mathbbm{1}_{A_1}+\delta_y P_{t_1}\mathbbm{1}_{A_2}).
$$
The continuity of $(x,y)\mapsto\bar{\mu}_{x,y}(f)$ on $E^2$ holds as above.

Fix $x_1\in E$ and $d_1>0$ such that $\inf_{x\in B(x_1,d_1)} f(x)>0$. Then~\eqref{eq:irreducibility-g} entails
\begin{align*}
  \bar{\mu}_{x_1,x_1}(f) & =\delta_{x_1}P_{t_1}f\geq\P_{x_1}(X_{t_1}\in B(x_1,d_1))\,\inf_{x\in B(x_1,d_1)}f(x)>0.
\end{align*}
Therefore, there exist $r_1,a_1>0$ such that $\bar{\mu}_{x,y}(f)\geq a_1$ for all $x,y\in B(x_1,r_1)$.

Hence, for all nonnegative measurable $g:E\rightarrow\R_+$ and for all $x,y\in E$ and all $u'\in E$,
\begin{align}
  \delta_x P_{2t_1}g & \geq \int_E\11_{u\in B(x_1,r_1)}P_{t_1}g(u)\,\delta_x P_{t_1}(du) \notag \\
  & \geq \int_E\11_{u,u'\in B(x_1,r_1)}\bar{\mu}_{u,u'}(g)\,\delta_x P_{t_1}(du). \label{eq:calcul}
\end{align}
Integrating both sides of the inequality w.r.t.\ $\delta_y P_{t_1}(du')$, we obtain
\begin{align*}
  \delta_x P_{2t_1}g &\geq  \delta_x P_{2t_1}g \delta_y P_{t_1}(E)\geq \iint_{E\times E}\11_{u,u'\in
    B(x_1,r_1)}\bar{\mu}_{u,u'}(g)\,\delta_x P_{t_1}(du)\,\delta_y P_{t_1}(du').
\end{align*}
Since this holds for all nonnegative measurable $g$ and since $\mu_{x,y}$ is the infimum measure between $\delta_x P_{2t_1}$ and
$\delta_y P_{2t_1}$, by symmetry, we have proved that
\begin{align*}
  \mu_{x,y}(\cdot) & \geq \iint_{E\times E}\11_{u,u'\in B(x_1,r_1)}\bar{\mu}_{u,u'}(\cdot)\,\delta_x P_{t_1}(du)\,\delta_y P_{t_1}(du').
\end{align*}
Therefore,~\eqref{eq:irreducibility-g} entails
\begin{equation*}
  \mu_{x,y}(f)\geq a_1\P_x(X_{t_1}\in B(x_1,r_1))\P_y(X_{t_1}\in B(x_1,r_1))>0. \qedhere
\end{equation*}
\end{proof}

We now construct the measure $\nu_{x,y}$ of Condition~(A1'). Using a similar computation as in~\eqref{eq:calcul} and integrating with
respect to $\delta_y P_{2t_1}(du')/\delta_y P_{2t_1}\11_E$, we obtain for all $x,y\in E$ and all nonnegative measurable
$f:E\rightarrow\R_+$
\begin{align*}
  \delta_x P_{4t_1}f&\geq
  \iint_{K\times K}  \mu_{u,u'}(f)\,\delta_x P_{2t_1}(du)\,\frac{\delta_y P_{2t_1}(du')}{\delta_y P_{2t_1}\11_E}.
\end{align*}
Since $\delta_x P_{4t_1}\11_E\leq \delta_x P_{2t_1}\11_E$,
\begin{align*}
  \frac{\delta_x P_{4t_1}f}{\delta_x P_{4t_1}\11_E}&\geq
  \iint_{K\times K} \mu_{u,u'}(f)\,\frac{\delta_x P_{2t_1}(du)}{\delta_x P_{2t_1}\11_E}\,\frac{\delta_y
    P_{2t_1}(du')}{\delta_y P_{2t_1}\11_E} \\ & = m_{x,y}\nu_{x,y}(f),
\end{align*}
where
\begin{align*}
  m_{x,y}:=\iint_{K\times K}\mu_{u,u'}(E)\,\frac{\delta_x P_{2t_1}(du)}{\delta_x P_{2t_1}\11_E}\,\frac{\delta_y
  P_{2t_1}(du')}{\delta_y P_{2t_1}\11_E}
\end{align*}
and
\begin{align}
  \label{eq:def-nu}
  \nu_{x,y}:=\frac{1}{m_{x,y}}\iint_{K\times K}\mu_{u,u'}(\cdot)\,\frac{\delta_x P_{2t_1}(du)}{\delta_x
    P_{2t_1}\11_E}\,\frac{\delta_y P_{2t_1}(du')}{\delta_y P_{2t_1}\11_E}.
\end{align}
Note that
\begin{align}
  m_{x,y} & \geq \inf_{u,u'\in K^2} \mu_{u,u'}(E) \iint_{K\times K}\frac{\delta_x P_{2t_1}(du)}{\delta_x
    P_{2t_1}\11_E}\,\frac{\delta_y P_{2t_1}(du')}{\delta_y P_{2t_1}\11_E} \notag \\ & \geq A^2 \inf_{u,u'\in K^2} \mu_{u,u'}(E)
  >0, \label{eq:lower-bound-mass-nu}
\end{align}
because of~\eqref{eq:step1} and Lemma~\ref{lem:positivity}. Hence the probability measure $\nu_{x,y}$ is well-defined and we have
proved (A1') for $t_0=4t_1$ and $c_1=A^2 \inf_{u,u'\in K^2} \mu_{u,u'}(E)$.

\subsubsection{Proof of (A2')}
\label{sec:proof-of-A2}

Our goal is now to prove Condition~(A2'). We first prove the following gradient estimate for $f=\mathbbm{1}_E$.

\begin{lem}
  \label{lem:wang-better}
  There exists a constant $C''>0$ such that, for all $t\geq 4t_1$,
  \begin{equation}
    \label{eq:wang-better}
    \|\nabla P_t \mathbbm{1}_E\|_\infty\leq C''\|P_t \mathbbm{1}_E\|_\infty.
  \end{equation}
\end{lem}

Note that, compared to~(\ref{eq:wang-g}), the difficulty is that we
replace $\|\11_E\|_\infty$ by the smaller
$\|P_{t_1}\11_E\|_\infty$ and that we extend this inequality to any time
$t$ large enough.

\begin{proof}
We first use~\eqref{eq:step1} to compute
\begin{align*}
  P_{4t_1}\mathbbm{1}_E(x) \geq \PP_x(X_{2t_1}\in K)\inf_{y\in K}\PP_y(2t_1<\tau_\d) 
  & \geq mA\PP_x(2t_1<\tau_\d),
\end{align*}
where $m:=\inf_{y\in K}\PP_y(2t_1<\tau_\d)$ is positive because of Lemma~\ref{lem:positivity}. Integrating the last inequality with
respect to $(\delta_yP_{t-4t_1})(dx)$ for any fixed $y\in E$ and $t\geq 4t_1$, we deduce that
$$
\|P_t \mathbbm{1}_E\|_\infty\geq m A \|P_{t-2t_1} \mathbbm{1}_E\|_\infty.
$$
Hence it follows from~\eqref{eq:gradEst} that, for all $t\geq 4t_1$,
\begin{align*}
  \|\nabla P_t\mathbbm{1}_E\|_\infty &=\|\nabla P_{t_1}(P_{t-t_1}\mathbbm{1}_E)\|_\infty
  \leq C \|P_{t-t_1}\mathbbm{1}_E\|_\infty\\
  & \leq C \|P_{t-2t_1}\mathbbm{1}_E\|_\infty \leq \frac{C}{mA}\|P_t \mathbbm{1}_E\|_\infty.
\end{align*}
This concludes the proof of Lemma~\ref{lem:wang-better}.
\end{proof}

This lemma implies that the function
\begin{align}
\label{eq:the-function}
h_t:x\in E\cup\{\d\} \mapsto \frac{P_t\mathbbm{1}_E(x)}{\|P_t\mathbbm{1}_E\|_\infty}
\end{align}
is $C''$-Lipschitz for all $t\geq 4t_1$. Since this function vanishes on $\d$ and its maximum over $E$ is 1, we deduce that, for any
$t\geq 4t_1$, there exists at least one point $z_t\in E$ such that $h_t(z_t)=1$. Since $h_t$ is $C''$-Lipschitz, we also deduce that $\rho(z_t,\d)\geq 1/C''$. Moreover, for all
$x\in E$,
\begin{align}
\label{eq:maj-step2}
\frac{P_t\mathbbm{1}_E(x)}{\|P_t\mathbbm{1}_E\|_\infty}\geq f_{z_t}(x),
\end{align}
 where, for all $z\in E$ and $x\in E$, $f_z(x)=\left( 1-C''\rho(x,z)\right)\vee 0$.
We define the compact set $K'=\{x\in E:\rho(x,\d)\geq 1/C''\}$ so that $z_t\in K'$ for all $t\geq 4t_1$. Then, for all $x,y\in E$ and
for all $t\geq 4t_1$, using the definition~\eqref{eq:def-nu} of $\nu_{x,y}$,
\begin{align*}
  \PP_{\nu_{x,y}}(t<\tau_\d) & \geq \|P_t\mathbbm{1}_E\|_\infty\,\nu_{x,y}(f_{z_t}) \\ 
  &  =\frac{\|P_t\mathbbm{1}_E\|_\infty}{m_{x,y}}\iint_{K\times K}\mu_{z,z'}(f_{z_t}) \,\frac{\delta_x
  P_{2t_1}(dz)}{\delta_x P_{2t_1}\11_E}\,\frac{\delta_y P_{2t_1}(dz')}{\delta_y P_{2t_1}\11_E}.
\end{align*}
Since $z\mapsto f_z$ is Lipschitz for the
  $\|\cdot\|_\infty$ norm (indeed,
  $|f_z(x)-f_{z'}(x)|\leq C''|\rho(x,z)-\rho(x,z')|\leq C''\rho(z,z')$
  for all $x,z,z'\in E^3$), it follows from Lemma~\ref{lem:positivity}
  that $(x,y,z)\mapsto \mu_{x,y}(f_{z})$ is positive and continuous on
  $E^3$. Hence
$c:=\inf_{x\in K,\,y\in K,\,z\in K'}\mu_{x,y}(f_{z})>0$ and, using
that $m_{x,y}\leq 1$,
\begin{align*}
  \PP_{\nu_{x,y}}(t<\tau_\d) & \geq c\|P_t\mathbbm{1}_E\|_\infty\iint_{K\times K}\frac{\delta_x P_{2t_1}(dz)}{\delta_x
    P_{2t_1}\11_E}\,\frac{\delta_y P_{2t_1}(dz')}{\delta_y P_{2t_1}\11_E} \geq c A^2\|P_t\mathbbm{1}_E\|_\infty,
\end{align*}
where the last inequality follows from~\eqref{eq:step1}.

This entails Condition~(A2') for all $t\geq 4t_1$. For $t\leq 4 t_1$, 
\begin{align*}
  \PP_{\nu_{x,y}}(t<\tau_\d) & \geq \PP_{\nu_{x,y}}(4t_1<\tau_\d)  \geq cA^2\|P_{4t_1}\11_E\|_\infty\\
  &\geq cA^2\|P_{4t_1}\11_E\|_\infty \sup_{z\in E} \P_z(t<\tau_\d)>0.
\end{align*}
This ends the proof of~(A2') and hence of~\eqref{eq:intro-expo-cv}.

\subsubsection{Contraction in total variation norm}
\label{sec:proof-of-contraction}

It only remains to prove~\eqref{eq:thm:general-criterion}. By~\eqref{eq:intro-contr}, we need to prove that there exists a constant $a>0$ such that, for all probability
measure $\mu$ on $E$,
\begin{align}
  \label{eq:last-goal}
  c(\mu)&:=\inf_{t\geq 0}\frac{\P_\mu(t<\tau_\d)}{\|P_t\11_E\|_\infty}\geq a\mu(\rho_\d).
\end{align}

Because of the equivalence between (A) and (A')~\cite[Theorem~2.1]{champagnat-villemonais-15}, enlarging $t_0$ and reducing $c_1 $
and $c_2$, one can assume without loss of generality that $\nu=\nu_{x,y}$ does not depend on $x,y\in E$. Then, using~(A1) and~(A2),
we deduce that, for all $t\geq t_0\geq 4t_1$,
\begin{align*}
  \P_\mu(t<\tau_\d)&=\mu(P_{t_0} P_{t-t_0} \11_E) \geq c_1\P_\mu(t_0<\tau_\d)\P_\nu(t-t_0<\tau_\d)\\
  &\geq c_1 c_2 \|P_{t-t_0}\11_E\|_\infty\P_\mu(t_0<\tau_\d)\geq c_1 c_2 \|P_{t}\11_E\|_\infty\P_\mu(t_0<\tau_\d).
\end{align*}
Now, using Assumption~\eqref{eq:majoration-g}, we deduce that
\begin{align*}
  \P_\mu(t_0<\tau_\d)\geq \E_\mu\left(\11_{T_K<t_1}\inf_{y\in K}\P_y(t_0<\tau_\d)\right)\geq C'\mu(\rho(\d,\cdot))\,\inf_{y\in K}\P_y(t_0<\tau_\d),
\end{align*}
where the constant $C'':=C' \inf_{y\in K}\P_y(t_0<\tau_\d)$ is positive. For $t\leq t_0$, the last inequality entails
\begin{align*}
  \P_\mu(t<\tau_\d)\geq \P_\mu(t_0<\tau_\d)\geq C''\mu(\rho(\d,\cdot))\geq C''\mu(\rho(\d,\cdot))\|P_t\11_E\|_\infty.
\end{align*}
Hence~\eqref{eq:last-goal} holds true with $a=c_1c_2 C''$. This ends the proof of Theorem~\ref{thm:general-criterion}.

\appendix

\section{Proof of~\eqref{eq:intro-contr}}
\label{sec:appendix}

Let us assume that Condition~(A) is satisfied. For all $t\geq 0$ and all probability measure $\pi$ on $E$, let $
c_t(\pi):=\frac{\pi(P_t\11_E)}{\|P_t\11_E\|_\infty}$. 
In the proof of~\cite[Corollary~2.2]{champagnat-villemonais-15}, it is proved that, for all probability measures $\pi_1,\pi_2$ on $E$
\begin{align*}
\left\|\P_{\pi_1}(X_t\in\cdot\mid t<\tau_\d)-\P_{\pi_2}(X_t\in\cdot\mid t<\tau_\d)\right\|_{TV} 
&\leq  \frac{(1-c_1c_2)^{\lfloor t/t_0\rfloor}}{c_t(\pi_1)\vee c_t(\pi_2)}\|\pi_1-\pi_2\|_{TV}.
\end{align*}
But 
\begin{align*}
\inf_{t\geq 0}c_t(\pi_1)\vee c_t(\pi_2)\geq (\inf_{t\geq 0} c_t(\pi_1))\vee (\inf_{t\geq 0} c_t(\pi_2))=c(\pi_1)\vee c(\pi_2).
\end{align*}
This ends the proof of~\eqref{eq:intro-contr}.

\bibliographystyle{abbrv}
\bibliography{biblio-bio,biblio-denis,biblio-math,biblio-math-nicolas}

\def\cprime{$'$} \def\cprime{$'$} \def\cprime{$'$} \def\cprime{$'$}
  \def\polhk#1{\setbox0=\hbox{#1}{\ooalign{\hidewidth
  \lower1.5ex\hbox{`}\hidewidth\crcr\unhbox0}}} \def\cprime{$'$}
  \def\lfhook#1{\setbox0=\hbox{#1}{\ooalign{\hidewidth
  \lower1.5ex\hbox{'}\hidewidth\crcr\unhbox0}}} \def\cprime{$'$}
  \def\cprime{$'$}
\begin{thebibliography}{10}

\bibitem{aronson}
D.~G. Aronson.
\newblock Non-negative solutions of linear parabolic equations.
\newblock {\em Ann. Scuola Norm. Sup. Pisa (3)}, 22:607--694, 1968.

\bibitem{bogdan-al-10}
K.~Bogdan, T.~Grzywny, and M.~Ryznar.
\newblock Heat kernel estimates for the fractional {L}aplacian with {D}irichlet
  conditions.
\newblock {\em Ann. Probab.}, 38(5):1901--1923, 2010.

\bibitem{CCLMMS09}
P.~Cattiaux, P.~Collet, A.~Lambert, S.~Mart{\'{\i}}nez, S.~M{\'e}l{\'e}ard, and
  J.~San~Mart{\'{\i}}n.
\newblock Quasi-stationary distributions and diffusion models in population
  dynamics.
\newblock {\em Ann. Probab.}, 37(5):1926--1969, 2009.

\bibitem{Cattiaux2008}
P.~Cattiaux and S.~M{\'e}l{\'e}ard.
\newblock Competitive or weak cooperative stochastic lotka-volterra systems
  conditioned to non-extinction.
\newblock {\em J. Math. Biology}, 60(6):797--829, 2010.

\bibitem{ChampagnatVillemonais2015c}
N.~{Champagnat} and D.~{Villemonais}.
\newblock {Uniform convergence of conditional distributions for absorbed
  one-dimensional diffusions}.
\newblock {\em ArXiv e-prints}, June 2015.

\bibitem{champagnat-villemonais-15}
N.~Champagnat and D.~Villemonais.
\newblock Exponential convergence to quasi-stationary distribution and
  {$Q$}-process.
\newblock {\em Probab. Theory Related Fields}, 164(1-2):243--283, 2016.

\bibitem{ChampagnatVillemonais2016U}
N.~{Champagnat} and D.~{Villemonais}.
\newblock {Uniform convergence to the Q-process}.
\newblock {\em {T}o appear in {E}lectron. {C}ommun. {P}rob.}, 2017.

\bibitem{chen-al-11}
Z.-Q. Chen, P.~Kim, and R.~Song.
\newblock Heat kernel estimates for {$\Delta+\Delta^{\alpha/2}$} in {$C^{1,1}$}
  open sets.
\newblock {\em J. Lond. Math. Soc. (2)}, 84(1):58--80, 2011.

\bibitem{chen-al-12}
Z.-Q. Chen, P.~Kim, and R.~Song.
\newblock Dirichlet heat kernel estimates for fractional {L}aplacian with
  gradient perturbation.
\newblock {\em Ann. Probab.}, 40(6):2483--2538, 2012.

\bibitem{chen-al-15}
Z.-Q. Chen, P.~Kim, and R.~Song.
\newblock Stability of {D}irichlet heat kernel estimates for non-local
  operators under {F}eynman-{K}ac perturbation.
\newblock {\em Trans. Amer. Math. Soc.}, 367(7):5237--5270, 2015.

\bibitem{davies-simon-84}
E.~B. Davies and B.~Simon.
\newblock Ultracontractivity and the heat kernel for {S}chr\"odinger operators
  and {D}irichlet {L}aplacians.
\newblock {\em J. Funct. Anal.}, 59(2):335--395, 1984.

\bibitem{DelMoralVillemonais2015}
P.~{Del Moral} and D.~{Villemonais}.
\newblock {Exponential mixing properties for time inhomogeneous diffusion
  processes with killing}.
\newblock {\em Bernoulli Journal}, 2016.
\newblock To appear.

\bibitem{Dynkin2002}
E.~B. Dynkin.
\newblock {\em Diffusions, superdiffusions and partial differential equations},
  volume~50 of {\em American Mathematical Society Colloquium Publications}.
\newblock American Mathematical Society, Providence, RI, 2002.

\bibitem{HeningKolb}
A.~{Hening} and M.~{Kolb}.
\newblock {Quasistationary distributions for one-dimensional diffusions with
  singular boundary points}.
\newblock {\em ArXiv e-prints}, Sept. 2014.

\bibitem{kallenberg-02}
O.~Kallenberg.
\newblock {\em Foundations of modern probability}.
\newblock Probability and its Applications (New York). Springer-Verlag, New
  York, second edition, 2002.

\bibitem{kim-kim-14}
K.-Y. Kim and P.~Kim.
\newblock Two-sided estimates for the transition densities of symmetric
  {M}arkov processes dominated by stable-like processes in {$C^{1,\eta}$} open
  sets.
\newblock {\em Stochastic Process. Appl.}, 124(9):3055--3083, 2014.

\bibitem{kim-song-07}
P.~Kim and R.~Song.
\newblock Estimates on {G}reen functions and {S}chr\"odinger-type equations for
  non-symmetric diffusions with measure-valued drifts.
\newblock {\em J. Math. Anal. Appl.}, 332(1):57--80, 2007.

\bibitem{KnoblochPartzsch2010}
R.~Knobloch and L.~Partzsch.
\newblock Uniform conditional ergodicity and intrinsic ultracontractivity.
\newblock {\em Potential Analysis}, 33:107--136, 2010.

\bibitem{Kolb2011}
M.~Kolb and A.~W{\"u}bker.
\newblock Spectral analysis of diffusions with jump boundary.
\newblock {\em J. Funct. Anal.}, 261(7):1992--2012, 2011.

\bibitem{le-gall}
J.-F. Le~Gall.
\newblock {\em Brownian motion, martingales, and stochastic calculus}, volume
  274 of {\em Graduate Texts in Mathematics}.
\newblock Springer, french edition, 2016.

\bibitem{lierl-saloffcoste-14}
J.~Lierl and L.~Saloff-Coste.
\newblock The {D}irichlet heat kernel in inner uniform domains: local results,
  compact domains and non-symmetric forms.
\newblock {\em J. Funct. Anal.}, 266(7):4189--4235, 2014.

\bibitem{lindvall-rogers-86}
T.~Lindvall and L.~C.~G. Rogers.
\newblock Coupling of multidimensional diffusions by reflection.
\newblock {\em Ann. Probab.}, 14(3):860--872, 1986.

\bibitem{Littin2012}
J.~Littin~C.
\newblock Uniqueness of quasistationary distributions and discrete spectra when
  {$\infty$} is an entrance boundary and 0 is singular.
\newblock {\em J. Appl. Probab.}, 49(3):719--730, 2012.

\bibitem{miura-14}
Y.~Miura.
\newblock Ultracontractivity for {M}arkov semigroups and quasi-stationary
  distributions.
\newblock {\em Stoch. Anal. Appl.}, 32(4):591--601, 2014.

\bibitem{Nualart2006}
D.~Nualart.
\newblock {\em The {M}alliavin calculus and related topics}.
\newblock Probability and its Applications (New York). Springer-Verlag, Berlin,
  second edition, 2006.

\bibitem{Priola2006}
E.~Priola and F.-Y. Wang.
\newblock Gradient estimates for diffusion semigroups with singular
  coefficients.
\newblock {\em J. Funct. Anal.}, 236(1):244--264, 2006.

\bibitem{stroock-varadhan-79}
D.~W. Stroock and S.~R.~S. Varadhan.
\newblock {\em Multidimensional diffusion processes}, volume 233 of {\em
  Grundlehren der Mathematischen Wissenschaften [Fundamental Principles of
  Mathematical Sciences]}.
\newblock Springer-Verlag, Berlin, 1979.

\bibitem{wang}
F.-Y. Wang.
\newblock Sharp explicit lower bounds of heat kernels.
\newblock {\em Ann. Probab.}, 25(4):1995--2006, 1997.

\bibitem{wang-04}
F.-Y. Wang.
\newblock Gradient estimates of dirichlet heat semigroups and application to
  isoperimetric inequalities.
\newblock {\em Ann. Probab.}, 32(1A):424--440, 01 2004.

\bibitem{zhang}
Q.~S. Zhang.
\newblock Gaussian bounds for the fundamental solutions of {$\nabla (A\nabla
  u)+B\nabla u-u_t=0$}.
\newblock {\em Manuscripta Math.}, 93(3):381--390, 1997.

\bibitem{zhang-02}
Q.~S. Zhang.
\newblock The boundary behavior of heat kernels of {D}irichlet {L}aplacians.
\newblock {\em J. Differential Equations}, 182(2):416--430, 2002.

\end{thebibliography}

\end{document}